\documentclass[10pt,reqno]{amsart}

\usepackage{amssymb}
\usepackage{amsmath}
\usepackage{amsfonts}
\usepackage{amscd}
\usepackage{mathrsfs}
\usepackage{color}

\newtheorem{theorem}{Theorem}
\newtheorem{corollary}[theorem]{Corollary}

\newtheorem{proposition}[theorem]{Proposition}

\newcommand{\C}{\mathbf{C}}
\newcommand{\R}{\mathbf{R}}

\begin{document}
\title[]{On Fermat Diophantine functional equations \\ and little Picard theorem}

\author[]{Jingbo Liu, Qi Han, and Wei Chen}

\address{\footnotesize{Department of Mathematics, The University of Hong Kong, Hong Kong, P.R. China
\vskip 2pt Department of Mathematics, Texas A\&M University, San Antonio, Texas 78224, USA
\vskip 2pt Department of Mathematics, Shandong University, Jinan, Shandong 250100, P.R. China
\vskip 2pt Email: {\sf jliu02@hku.hk} (Jingbo Liu) {\sf qhan@tamusa.edu} (Qi Han) {\sf weichensdu@126.com} (Wei Chen)}}



\begin{abstract}
We discuss equivalence conditions on the non-existence of non-trivial meromorphic solution to the Fermat Diophantine equation $f^m(z)+g^n(z)=1$ with integers $m,n\geq2$, from which other approaches to prove little Picard theorem are described.
\end{abstract}

\maketitle





This paper is primarily devoted to the description on equivalence conditions concerning the non-existence of non-constant meromorphic solution to the equation
\begin{equation}\label{Eq1}
f^m(z)+g^n(z)=1
\end{equation}
over the complex plane $\C$, where $m,n\geq2$ are positive integers.

\vskip 2pt
It seems to us that Montel first studied the functional analog \eqref{Eq1} to the Fermat Diophantine equation $x^p+y^p=1$; see Jategaonkar \cite{Ja} for an elementary proof written in English.
Later work has been discussed in Gross \cite{Gr} and Baker \cite{Ba}, where full characterization of non-constant meromorphic solutions to \eqref{Eq1} are provided when $m=n\geq2$.
In general, \eqref{Eq1} has no non-trivial entire solution provided $m+n<mn$ that follows from a theorem of Cartan \cite[section 4]{GH}; see also Toda \cite{To}.
For meromorphic solutions, it seems to us that this is due to Li \cite[section 4]{Li2}.
For convenience of the reader, we summarize those renown results below.

\begin{proposition}\label{Pr1}
The functional equation \eqref{Eq1} has non-trivial meromorphic solutions $f$ and $g$ over $\C^k$ for $k\geq1$ if and only if when \\
{\bf(I)} $m=n=2$, and $\displaystyle f=\frac{1-\alpha^2}{1+\alpha^2}$ and $\displaystyle g=\frac{2\hspace{0.2mm}\alpha}{1+\alpha^2}$, for a non-constant meromorphic function $\alpha$, are the only solutions; \\
{\bf(II)} $m=n=3$, and $\displaystyle f=\frac{3+\sqrt{3}\hspace{0.2mm}\wp'(\beta)}{6\hspace{0.2mm}\wp(\beta)}$ and
$\displaystyle g=\eta\hspace{0.2mm}\frac{3-\sqrt{3}\hspace{0.2mm}\wp'(\beta)}{6\hspace{0.2mm}\wp(\beta)}$, for some non-constant entire function $\beta$, are the only solutions, where $\eta^3=1$ and $\wp$ denotes the Weierstrass $\wp$-function that satisfies $(\wp')^2\equiv4\hspace{0.2mm}\wp^3-1$ after appropriately choosing its periods; \\
{\bf(III)} $m=2$ and $n=3$, and $\displaystyle f=\mathrm{i}\hspace{0.2mm}\wp'(\beta)$ and $\displaystyle g=\eta\sqrt[3]{4}\hspace{0.2mm}\wp(\beta)$, for some non-constant entire function $\beta$, generate a pair of solutions with $\wp$ satisfying $(\wp')^2\equiv4\hspace{0.2mm}\wp^3-1$; \\
{\bf(IV)} $m=2$ and $n=4$, and
$\displaystyle f_1=\frac{-4\hspace{0.2mm}\wp^3(\beta)+\frac{1}{12}\hspace{0.2mm}\wp(\beta)+\frac{1}{3}}{4\hspace{0.2mm}\wp^3(\beta)+\frac{1}{12}\hspace{0.2mm}\wp(\beta)+\frac{1}{6}}$
and $\displaystyle g_1=2\hspace{0.2mm}\zeta\hspace{0.2mm}\frac{\wp(\beta)}{\wp'(\beta)}$ as well as
$\displaystyle f_2=\mathrm{i}\hspace{0.2mm}\frac{4\hspace{0.2mm}\wp^3(\beta)-\frac{1}{12}\hspace{0.2mm}\wp(\beta)-\frac{1}{3}}{4\hspace{0.2mm}\wp^2(\beta)}$
and $\displaystyle g_2=\zeta\hspace{0.2mm}\frac{\mathrm{i}\hspace{0.2mm}\wp'(\beta)}{2\hspace{0.2mm}\wp(\beta)}$, for some non-constant entire function $\beta$, generate pairs of solutions with $\zeta^4=1$ and $\wp$ satisfying $(\wp')^2\equiv4\hspace{0.2mm}\wp^3+\frac{1}{12}\wp+\frac{1}{6}$; \\
{\bf(V)} $m=3$ and $n=2$, and this follows from {\bf (III)} by symmetry; \\
{\bf(VI)} $m=4$ and $n=2$, and this follows from {\bf (IV)} by symmetry.
\end{proposition}

{\sl Remark.}
$f=\sin(\beta)$ and $g=\cos(\beta)$ in {\bf(I)}, through $\alpha=\tan\bigl(\frac{\beta}{2}\bigr)$ with a non-constant entire function $\beta$, are the only entire solutions to the functional equation \eqref{Eq1}.
Moreover, in {\bf(II)}, the constant ``$-1$'' is not essential - we simply need to attain ``$g_2=0$'' in which case all Weierstrass $\wp$-functions (viewed as tori over $\C$) are isomorphic.
The uniqueness of meromorphic solutions to \eqref{Eq1} in {\bf(I)} and {\bf(II)} follows from Coman and Poletsky \cite[theorem 5.2]{CP} that further says Han \cite{Ha} indeed provided full characterization of meromorphic solutions to the PDE $u^m_{z_1}+u^m_{z_2}=u^m$ (recall that \cite{Ha} was written before \cite{CP}).
The results in {\bf (IV)}, and {\bf (VI)}, follow from Huber \cite[equations (4) and (7)-(7.3)]{Hu} by setting $b_0=1$, $b_1=b_2=b_3=0$ and $b_4=-1$ - the result in Li \cite[page 217]{Li2} (citing \cite{Hu}) by taking $f_1=\wp'$ is unfortunately not correct.

\vskip 2pt
It is thus natural to seek equivalence conditions concerning the non-existence of non-trivial meromorphic solution to \eqref{Eq1}, and we will consider this question in this paper.

\vskip 2pt
Our first elementary result makes use of the condition $\mathcal{Z}(f)=\mathcal{Z}(g)$ ignoring multiplicity, where $\mathcal{Z}(h)$ represents the zero set of the function $h$.
Actually, we observe
\begin{equation}\label{Eq2}
f^m(z)-1=-\hspace{0.2mm}g^n(z) \hspace{2mm} \mathrm{and} \hspace{2mm} g^n(z)-1=-\hspace{0.2mm}f^m(z).
\end{equation}
Let $\mu_0,\mu_1,\ldots,\mu_{m-1}$ and $\nu_0,\nu_1,\ldots,\nu_{n-1}$ be respectively the distinct $m$-th and $n$-th roots of unity.
Then, this condition implies $f$ omits $0,\mu_0,\mu_1,\ldots,\mu_{m-1}$ and $g$ omits $0,\nu_0,\nu_1,\ldots,\nu_{n-1}$, so that little Picard theorem may be applied to show $f$ and $g$ are constant.

\begin{theorem}\label{Th1}
Meromorphic solutions $f$ and $g$ to the functional equation \eqref{Eq1} in $\C^k$ for $k\geq1$ are constant if and only if $\mathcal{Z}(f)=\mathcal{Z}(g)$ ignoring multiplicity when $\max\left\{m,n\right\}=2$, $\mathcal{Z}(f)\subseteq\mathcal{Z}(g)$ ignoring multiplicity when $n\geq3$, and $\mathcal{Z}(f)\supseteq\mathcal{Z}(g)$ ignoring multiplicity when $m\geq3$.
\end{theorem}

This result in turn implies little Picard theorem.
In fact, let $H$ be a meromorphic function omits $3$ distinct values; without loss of generality, assume the omitted values are $0,1,\infty$.
Thus, $H=e^{\gamma}$ and $1-H=e^{\delta}$.
So, $f:=e^{\frac{\gamma}{m}}$ and $g:=e^{\frac{\delta}{n}}$ satisfy \eqref{Eq1} with $\mathcal{Z}(f)=\mathcal{Z}(g)=\emptyset$.
Hence, $f$ is constant and so is $H$ - that is, little Picard theorem follows from theorem \ref{Th1}.
Also, notice that $f=\frac{1}{1+e^z}$ and $g=\frac{e^z}{1+e^z}$ with $\mathcal{Z}(f)=\mathcal{Z}(g)=\emptyset$ satisfy \eqref{Eq1} for $m=n=1$.

\vskip 2pt
Li \cite{Li1,LY,Li3} exploited the condition $\mathcal{Z}(f_{z_1})=\mathcal{Z}(g_{z_2})$ counting multiplicity and completely answered this question in $\C^2$ when $m,n\geq2$.
He then used those results to study meromorphic solutions to Fermat-type PDEs; see also Han \cite{Ha} and the references therein.

\vskip 2pt
This condition $\mathcal{Z}(f')=\mathcal{Z}(g')$ counting multiplicity can also be applied to answer our question over $\C$ easily when $m,n\geq2$.
Actually, one realizes by differentiation
\begin{equation}\label{Eq3}
mf^{m-1}f'=-\hspace{0.2mm}n\hspace{0.2mm}g^{n-1}g'.
\end{equation}
When $f(z_1)=\mu_j$, it has multiplicity $n\ell$ with $\ell\geq1$ and hence $f'(z_1)=0$, which cannot be true from \eqref{Eq3}.
So, $f$ omits $\mu_j$ and $g$ in turn omits $0$.
When $m\geq3$, little Picard theorem says $f$ is constant; otherwise, consider $g$ and $\nu_l$ to see $f$ omits $0$ - again, $f$ is constant.

\vskip 2pt
Recall (\cite[remark 2.2]{Li5}) $f=\frac{1-e^{2z}}{1+e^{2z}}$ and
$g=\frac{2\hspace{0.1mm}e^z}{1+e^{2z}}$ satisfy \eqref{Eq1} with $\mathcal{Z}(f')\varsubsetneq\mathcal{Z}(g')$ counting multiplicity for $m=n=2$.
On the other hand, note $f=e^z$ and $g=1-e^{mz}$ satisfy \eqref{Eq1} with $\mathcal{Z}(f')=\mathcal{Z}(g')=\emptyset$ for $m\geq1$ and $n=1$; the same result follows by symmetry for $m=1$ and $n\geq1$.
Therefore, all these preceding observations yield the following result.

\begin{theorem}\label{Th2}
Meromorphic solutions $f$ and $g$ to the functional equation \eqref{Eq1} in $\C$ are constant if and only if $\mathcal{Z}(f')=\mathcal{Z}(g')$ counting multiplicity when $m,n\geq2$.
\end{theorem}

A refinement of theorem \ref{Th2} is formulated as the following result.

\begin{theorem}\label{Th3}
Meromorphic solutions $f$ and $g$ to the functional equation \eqref{Eq1} in $\C$ are constant if and only if $\mathcal{Z}(f')=\mathcal{Z}(g')$ counting multiplicity when $m=n=2$, whereas $\mathcal{Z}(f')\subseteq\mathcal{Z}(g')$ or $\mathcal{Z}(f')\supseteq\mathcal{Z}(g')$ ignoring multiplicity when $m,n\geq2$ yet $\left(m,n\right)\neq\left(2,2\right)$.
\end{theorem}

This result is motivated by Li \cite[theorem 2.1]{Li5} where entire solutions to \eqref{Eq1} are considered when $m=n=2$, but the idea of our proof follows from Li \cite[theorem 1.1]{Li3} where Nevanlinna theory is used.
It is noteworthy that neither this problem nor the solution to it should depend on Nevanlinna theory, as seen from our descriptions regarding theorems \ref{Th1} and \ref{Th2}.

\vskip 2pt
This result implies little Picard theorem as well.
Actually, let $H$ be a meromorphic function omitting $0,1,\infty$.
Then, $H=e^{\gamma}$ and $1-H=e^{\delta}$, and $f:=e^{\frac{\gamma}{m}}$ and $g:=e^{\frac{\delta}{n}}$ satisfy \eqref{Eq1}.
In view of \eqref{Eq3}, $\mathcal{Z}(f')=\mathcal{Z}(g')$ counting multiplicity trivially (since $\mathcal{Z}(f)=\mathcal{Z}(g)=\emptyset$).
Via theorem \ref{Th3}, $f$ is constant and so is $H$ - that is, little Picard theorem follows from theorem \ref{Th3}.

\vskip 2pt
Picard's results are far-reaching even nowadays, and there is an extensive literature closely related to them.
Confining our attention only to little Picard theorem over $\C$, one finds many interesting results; just name a few recent ones, \cite{Be,Le,Li4,Li5,Sc,Si,Sk,Zh} etc.

\vskip 2pt
Below, we assume the familiarity with the basics of Nevanlinna theory \cite{Ne} of meromorphic functions in $\C$ such as the first main theorem and the logarithmic derivative lemma, and the standard notations such as the characteristic function $T(r,f)$, the proximity function $m(r,f)$ and the counting function $N(r,f)$.
$S(r,f)$ denotes any quantity satisfying $S(r,f)=o\left(T(r,f)\right)$ when $r\to\infty$, possibly outside of some set of $\R^+$ having finite Lebesgue measure.

\begin{proof}[Proof of Theorem \ref{Th3}]
The first case was discussed in theorem \ref{Th2}.
By symmetry, we only need to prove the other case if $\mathcal{Z}(f')\subseteq\mathcal{Z}(g')$ when $\left(m,n\right),\left(n,m\right)\in\left\{\left(2,3\right),\left(2,4\right),\left(3,3\right)\right\}$.

\vskip 2pt
For meromorphic solutions $f$ and $g$ to \eqref{Eq1}, write
\begin{equation}
H_0:=\frac{f'(g')^2}{(f^m-1)(g^n-1)}.\nonumber
\end{equation}
When $f(z_1)=\mu_j$ for $j=0,1,\ldots,m-1$, then $g(z_1)=0$ by \eqref{Eq2}.
If $g(z_1)=0$ has multiplicity $p$, $f(z_1)=\mu_j$ has multiplicity $np$, so that $f'(z_1)=0$ has multiplicity $np-1$.
Via our hypothesis, $g'(z_1)=0$ has multiplicity $p-1$ and $p\geq2$.
So, $H_0(z_1)=0$ has multiplicity $2p-3>0$.
Also, one immediately observes that $H^2_0(z_1)=0$ has multiplicity $4p-6\geq p$ - that is,
\begin{equation}\label{Eq4}
N\Bigl(r,\frac{1}{g}\Bigr)\leq N\Bigl(r,\frac{1}{H^2_0}\Bigr)+O(1).
\end{equation}
When $g(z^*_1)=\nu_l$ for $l=0,1,\ldots,n-1$, then $f(z^*_1)=0$ by \eqref{Eq2}.
If $f(z^*_1)=0$ has multiplicity $q$, $g(z^*_1)=\nu_l$ has multiplicity $mq$, so that $g'(z^*_1)=0$ has multiplicity $mq-1\geq2q-1$.
Therefore, $H_0(z^*_1)=0$ has multiplicity $mq-2\geq0$ - that is, $H_0$ is analytic at the point $z^*_1$.

\vskip 2pt
It is worth to notice that the above proof goes through whenever $m,n\geq2$.

\vskip 2pt
Next, when $f(z_{\infty})=g(z_{\infty})=\infty$ have multiplicities $s,t$ respectively, then $H_0(z_{\infty})=0$ has multiplicity $(m-1)\hspace{0.1mm}s+(n-2)\hspace{0.1mm}t-3\geq0$ - that is, $H_0$ is analytic at the point $z_{\infty}$.
In fact, note $ms=nt$.
If $\left(m,n\right)=\left(2,3\right)$, one has $s\geq3$ and $t\geq2$, so that $s+t\geq5$; if $\left(m,n\right)=\left(2,4\right)$, one has $s\geq2$ and $t\geq1$, so that $s+2\hspace{0.2mm}t\geq4$; if $\left(m,n\right)=\left(3,2\right)$, one has $s\geq2$ and $t\geq3$, so that $2\hspace{0.2mm}s\geq4$; if $\left(m,n\right)=\left(4,2\right)$, one has $s\geq1$ and $t\geq2$, so that $3\hspace{0.2mm}s\geq3$; finally, if $\left(m,n\right)=\left(3,3\right)$, one has $s=t\geq1$ and $2\hspace{0.2mm}s+t=3\hspace{0.2mm}s=3\hspace{0.2mm}t\geq3$.
Hence, $H_0$ cannot have a pole at $z_{\infty}$.

\vskip 2pt
All these preceding discussions lead to that $H_0$ is an entire function.
Using the logarithmic derivative lemma, it follows that
\begin{equation}\label{Eq5}
T(r,H_0)=m(r,H_0)+O(1)=S(r),
\end{equation}
where $S(r):=S(r,f)=S(r,g)$ because $m\hspace{0.2mm}T(r,f)=n\hspace{0.2mm}T(r,g)+O(1)$ follows from equation \eqref{Eq1}.
Through the logarithmic derivative lemma again, one has
\begin{equation}
m\Bigl(r,\frac{1}{g}\Bigr)\leq m\Bigl(r,\frac{H_0}{g}\Bigr)+m\Bigl(r,\frac{1}{H_0}\Bigr)+O(1)=m\Bigl(r,\frac{1}{H_0}\Bigr)+S(r),\nonumber
\end{equation}
which combined with \eqref{Eq4}, \eqref{Eq5} and the first main theorem yields $T(r,g)=S(r)$.
Thus, $g$ must be constant and so is $f$ - that is, \eqref{Eq1} cannot have non-trivial meromorphic solutions.
\end{proof}

Many interesting results on meromorphic solutions to differential equations can be found in Hille \cite{Hi}; see also \cite{HL,HY}.
Our discussions lead to a non-existence result of non-constant meromorphic solutions to certain non-linear differential equations which otherwise can be difficult to handle using those available methods from differential equations.

\begin{corollary}\label{Co1}
Given an entire function $h$, meromorphic solutions $f$ and $g$ in $\C$ to
\begin{equation}\label{Eq6}
f^m+h^n(f')^{\ell n}=1
\end{equation}
are constant provided that $\ell,n,m\geq2$ but $\left(m,n\right)\neq\left(2,2\right)$.
\end{corollary}

The proof is standard verifying the latter case in theorem \ref{Th3} with $g:=h\hspace{0.1mm}(f')^{\ell}$.
On the other hand, note that $f=\frac{1-e^{2z}}{1+e^{2z}}$ satisfies \eqref{Eq6} with $h=\frac{1+e^{2z}}{2\hspace{0.1mm}e^z}$ for $\ell=1$ and $m=n=2$.

\vskip 2pt
Finally, we discuss some variants of \eqref{Eq1} as the following functional equations
\begin{equation}\label{Eq7}
f^2-2\hspace{0.2mm}\rho fg+g^2=1 \hspace{2mm} \mathrm{and} \hspace{2mm} f^3-3\hspace{0.2mm}\tau fg+g^3=1
\end{equation}
that are of independent interests.
In particular, the latter case relates to the classical work by Dixon \cite{Di}; see also the paper of Saleeby \cite{Sa} on meromorphic solutions to PDEs.

\begin{theorem}\label{Th4}
Let $\rho^2\neq1$ and $\tau^3\neq-1$.
Meromorphic solutions $f$ and $g$ over $\C$ to the functional equations in \eqref{Eq7} are constant if and only if \\
{\bf(A)} Quadratic case: $\mathcal{Z}(f)=\mathcal{Z}(g)$ ignoring multiplicity or $\mathcal{Z}(f')=\mathcal{Z}(g')$ counting multiplicity, \\
{\bf(B)} Cubic case: $\mathcal{Z}(f)=\mathcal{Z}(g)$ or $\mathcal{Z}(f')=\mathcal{Z}(g')$ both ignoring multiplicity.
\end{theorem}

\begin{proof}
For the quadratic case, one simply follows Li \cite[remark 1.2]{Li5} and observes $f+\rho_1g=h$ and $f+\rho_2g=h^{-1}$ for a meromorphic function $h$, so that it leads to
\begin{equation}
f=\frac{h^2-\frac{\rho_1}{\rho_2}}{\Bigl(1-\frac{\rho_1}{\rho_2}\Bigr)h} \hspace{2mm} \mathrm{and} \hspace{2mm} g=\frac{h^2-1}{(\rho_1-\rho_2)h},\nonumber
\end{equation}
where $\rho_1=\rho+\sqrt{\rho^2-1}$ and $\rho_2=\rho-\sqrt{\rho^2-1}$.
Since we assume that $\mathcal{Z}(f)=\mathcal{Z}(g)$ ignoring multiplicity, $h$ omits $4$ values $\pm\sqrt{\frac{\rho_1}{\rho_2}},\pm1$ and must be constant.
Besides, one has
\begin{equation}
f'=\frac{h'\Bigl(h^2+\frac{\rho_1}{\rho_2}\Bigr)}{\Bigl(1-\frac{\rho_1}{\rho_2}\Bigr)h^2} \hspace{2mm} \mathrm{and} \hspace{2mm} g'=\frac{h'(h^2+1)}{(\rho_1-\rho_2)h^2},\nonumber
\end{equation}
so that, as $\mathcal{Z}(f')=\mathcal{Z}(g')$ counting multiplicity, $h$ omits $4$ values $\pm\hspace{0.2mm}\mathrm{i}\sqrt{\frac{\rho_1}{\rho_2}}$ and $\pm\hspace{0.2mm}\mathrm{i}$.

\vskip 2pt
For the cubic case, recall \cite[page 559]{Sa}, for some entire function $\beta$ over $\C$,
\begin{equation}\label{Eq8}
f=\frac{-3\hspace{0.2mm}\tau\sqrt[3]{4}\hspace{0.2mm}\wp(\beta)+36+9\hspace{0.2mm}\tau^3+\wp'(\beta)}{6\hspace{0.2mm}\{\sqrt[3]{4}\hspace{0.2mm}\wp(\beta)+9\hspace{0.2mm}\tau^2\}}
\hspace{2mm} \mathrm{and} \hspace{2mm}
g=\frac{-3\hspace{0.2mm}\tau\sqrt[3]{4}\hspace{0.2mm}\wp(\beta)+36+9\hspace{0.2mm}\tau^3-\wp'(\beta)}{6\hspace{0.2mm}\{\sqrt[3]{4}\hspace{0.2mm}\wp(\beta)+9\hspace{0.2mm}\tau^2\}}
\end{equation}
are the only solutions to \eqref{Eq7}; see also our remark for proposition \ref{Pr1}.
Here, we assume
\begin{equation}\label{Eq9}
(\wp')^2\equiv4\hspace{0.2mm}\wp^3+27\hspace{0.2mm}\tau\sqrt[3]{4}\hspace{0.2mm}(8-\tau^3)\wp+54\hspace{0.2mm}(\tau^6+20\hspace{0.2mm}\tau^3-8).
\end{equation}
Notice that the {\sl modular discriminant} $\Delta$ of \eqref{Eq9},
\begin{equation}
\begin{aligned}
\Delta(\tau):=&\,\{-\hspace{0.2mm}27\hspace{0.2mm}\tau\sqrt[3]{4}\hspace{0.2mm}(8-\tau^3)\}^3-27\hspace{0.2mm}\{54\hspace{0.2mm}(\tau^6+20\hspace{0.2mm}\tau^3-8)\}^2 \\
=&-5038848\hspace{0.2mm}(\tau^9+3\hspace{0.2mm}\tau^6+3\hspace{0.2mm}\tau^3+1)=-\hspace{0.2mm}5038848\hspace{0.2mm}(\tau^3+1)^3,\nonumber
\end{aligned}
\end{equation}
as a function of $\tau$ vanishes only when $\tau^3=-1$.

\vskip 2pt
When we assume $\mathcal{Z}(f)=\mathcal{Z}(g)$ ignoring multiplicity, then $\tau\sqrt[3]{4}\hspace{0.2mm}\wp-12-3\hspace{0.2mm}\tau^3=0$ and $\wp'=0$ simultaneously, so that $4\hspace{0.2mm}\wp^3+27\hspace{0.2mm}\tau\sqrt[3]{4}\hspace{0.2mm}(8-\tau^3)\wp+54\hspace{0.2mm}(\tau^6+20\hspace{0.2mm}\tau^3-8)=0$ seeing \eqref{Eq9}.
If $\tau=0$, a contradiction follows immediately.
Otherwise, it leads to $\tau^6+2\hspace{0.2mm}\tau^3+1=0$ through a routine computation, which contradicts against our hypothesis.
Therefore,
\begin{equation}
\begin{aligned}
H_1:=&\,4\hspace{0.2mm}\wp^3(\beta)+27\hspace{0.2mm}\tau\sqrt[3]{4}\hspace{0.2mm}(8-\tau^3)\wp(\beta)+54\hspace{0.2mm}(\tau^6+20\hspace{0.2mm}\tau^3-8)\\
&-9\hspace{0.2mm}\{-\hspace{0.2mm}\tau\sqrt[3]{4}\hspace{0.2mm}\wp(\beta)+12+3\hspace{0.2mm}\tau^3\}^2=O\bigl(\wp^3(\beta)\bigr)\nonumber
\end{aligned}
\end{equation}
as an elliptic function never vanishes, which is impossible unless $\beta$ is constant.

\vskip 2pt
On the other hand, it is straightforward to derive that
\begin{equation}
\begin{aligned}
f'=\frac{-36\sqrt[3]{4}\hspace{0.2mm}(\tau^3+1)\wp'(\beta)-\sqrt[3]{4}\hspace{0.2mm}\{\wp'(\beta)\}^2+\{\sqrt[3]{4}\hspace{0.2mm}\wp(\beta)+9\hspace{0.2mm}\tau^2\}\wp''(\beta)}
{6\hspace{0.2mm}\{\sqrt[3]{4}\hspace{0.2mm}\wp(\beta)+9\hspace{0.2mm}\tau^2\}^2}\cdot\beta',\\
g'=\frac{-36\sqrt[3]{4}\hspace{0.2mm}(\tau^3+1)\wp'(\beta)+\sqrt[3]{4}\hspace{0.2mm}\{\wp'(\beta)\}^2-\{\sqrt[3]{4}\hspace{0.2mm}\wp(\beta)+9\hspace{0.2mm}\tau^2\}\wp''(\beta)}
{6\hspace{0.2mm}\{\sqrt[3]{4}\hspace{0.2mm}\wp(\beta)+9\hspace{0.2mm}\tau^2\}^2}\cdot\beta'.\nonumber
\end{aligned}
\end{equation}
Suppose there is a zero of $f'$ that is not a zero of $\beta'$.
Since we assume $\mathcal{Z}(f')=\mathcal{Z}(g')$ ignoring multiplicity, it induces $\{\sqrt[3]{4}\hspace{0.2mm}\wp+9\hspace{0.2mm}\tau^2\}\wp''=0$ and $\wp'=0$ simultaneously, so that using \eqref{Eq9} again $4\hspace{0.2mm}\wp^3+27\hspace{0.2mm}\tau\sqrt[3]{4}\hspace{0.2mm}(8-\tau^3)\wp+54\hspace{0.2mm}(\tau^6+20\hspace{0.2mm}\tau^3-8)=0$.
If $\wp=-\frac{9\hspace{0.2mm}\tau^2}{\sqrt[3]{4}}$, one has $\tau^6+2\hspace{0.2mm}\tau^3+1=0$ via a routine calculation; a contradiction.
If $\wp''=0$, one has $\wp^2=-\frac{27}{12}\hspace{0.2mm}\tau\sqrt[3]{4}\hspace{0.2mm}(8-\tau^3)$ since
\begin{equation}\label{Eq10}
\wp''\equiv6\hspace{0.2mm}\wp^2+\frac{27}{2}\hspace{0.2mm}\tau\sqrt[3]{4}\hspace{0.2mm}(8-\tau^3).
\end{equation}
As $\{4\hspace{0.2mm}\wp^3+27\hspace{0.2mm}\tau\sqrt[3]{4}\hspace{0.2mm}(8-\tau^3)\wp\}^2=\{-\hspace{0.2mm}54\hspace{0.2mm}(\tau^6+20\hspace{0.2mm}\tau^3-8)\}^2$, it yields $\tau^9+3\hspace{0.2mm}\tau^6+3\hspace{0.2mm}\tau^3+1=0$ from a direct calculation that leads to a contradiction.
Thus, by symmetry, \eqref{Eq9} and \eqref{Eq10},
\begin{equation}
\begin{aligned}
H_2:=&\,\bigl\{\sqrt[3]{4}\hspace{0.2mm}\{\wp'(\beta)\}^2-\{\sqrt[3]{4}\hspace{0.2mm}\wp(\beta)+9\hspace{0.2mm}\tau^2\}\wp''(\beta)\bigr\}^2
-\{36\sqrt[3]{4}\hspace{0.2mm}(\tau^3+1)\wp'(\beta)\}^2\\
=&\,\bigl\{2\sqrt[3]{4}\hspace{0.2mm}\wp^3(\beta)+O\bigl(\wp^2(\beta)\bigr)\bigr\}^2+O\bigl(\wp^3(\beta)\bigr)=O\bigl(\wp^6(\beta)\bigr)\nonumber
\end{aligned}
\end{equation}
can only vanish at the zeros of $\beta'$, which won't happen unless $\beta$ is constant.

\vskip 2pt
It is worthwhile to mention our preceding analyses used the fact that every elliptic function having no pole must be constant and every non-constant elliptic function has no Picard value; see Koecher and Krieg \cite{KK}.
As $\wp(\beta)$ has poles of multiplicity $2\hspace{0.2mm}\ell$ (for non-constant $\beta$), $H_1,H_2$ will admit all finite values $a\in\C$ $6\hspace{0.2mm}\ell,12\hspace{0.2mm}\ell$ times respectively in each of $\wp$'s fundamental domains (parallelograms) with $1$ possible exception - the possible finite Picard value of $\beta$.
This explains the contradiction on $H_1$; from a classical result of Clunie \cite{Cl}, we recognize $T(r,\beta)=S(r,\wp(\beta))$ that explains the contradiction on $H_2$ because $N\bigl(r,\frac{1}{\wp(\beta)-a}\bigr)=\Omega\left(T(r,\wp(\beta))\right)$.
\end{proof}





\end{document}